\documentclass[a4paper, 10pt, final]{amsart}

\usepackage{amsthm}

\theoremstyle{plain}

\usepackage[latin1]{inputenc}
\usepackage[T1]{fontenc}
\usepackage[english]{babel}
\usepackage{amssymb}

\usepackage{lmodern}
\usepackage[OT2,T1]{fontenc}

\usepackage{amsmath, amsthm, amsfonts}

\DeclareMathOperator{\weight}{weight}

\DeclareMathOperator{\Frob}{Frob}

\newtheorem{Theorem}{Theorem}
\newtheorem{Proposition}[Theorem]{Proposition}

\newtheorem{Lemma}[Theorem]{Lemma}

\theoremstyle{definition}
\newtheorem{Definition}[Theorem]{Definition}

\newtheorem{Example}[Theorem]{Example}

\newtheorem{Proposition-Definition}[Theorem]{Proposition-Definition}

\newtheorem{Remark}[Theorem]{Remark}
\newtheorem{Nota Bene}[Theorem]{Nota Bene}

\numberwithin{equation}{section}

\input cyracc.def
\DeclareFontFamily{U}{russian}{}
\DeclareFontShape{U}{russian}{m}{n}
        { <5><6> wncyr5
        <7><8><9> wncyr7
        <10><10.95><12><14.4><17.28><20.74><24.88> wncyr10 }{}
\DeclareSymbolFont{Russian}{U}{russian}{m}{n}
\DeclareSymbolFontAlphabet{\mathcyr}{Russian}
\makeatletter
\let\@math@cyr\mathcyr
\renewcommand{\mathcyr}[1]{\@math@cyr{\cyracc #1}}
\makeatother

\makeatletter
\newcounter{subsubsubsection}[subsubsection]
\renewcommand\thesubsubsubsection{\thesubsubsection .\@alph\c@subsubsubsection}
\newcommand\subsubsubsection{\@startsection{subsubsubsection}{4}{\z@}%
                                     {-3.25ex\@plus -1ex \@minus -.2ex}%
                                     {1.5ex \@plus .2ex}%
                                     {\normalfont\normalsize\bfseries}}
\newcommand*\l@subsubsubsection{\@dottedtocline{3}{10.0em}{4.1em}}
\newcommand*{\subsubsubsectionmark}[1]{}
\makeatother

\usepackage[top=2.5cm, bottom=2.5cm, left=2.5cm, right=2.5cm]{geometry}

\author{Hidekazu Furusho and David Jarossay}

\title{$p$-adic multiple $L$-functions and cyclotomic multiple harmonic values}

\address{Graduate School of Mathematics, Nagoya University, Furo-cho, Chikusa-ku, Nagoya 464-8602, Japan}

\email{furusho@math.nagoya-u.ac.jp}

\address{Universit\'{e} de Gen\`{e}ve, Section de math\'{e}matiques, 2-4 rue du Li\`{e}vre,
	Case postale 64,
	1211 Gen\`{e}ve, Switzerland}
\email{david.jarossay@unige.ch}

\begin{document}

\maketitle

\begin{abstract} 
We show that the special values at tuples of positive integers of the $p$-adic multiple $L$-function introduced by the first-named author et al. can be expressed in terms of the cyclotomic multiple harmonic values introduced by the second-named author.
\end{abstract}

\setcounter{section}{-1}

\section{Introduction}\label{Introduction}
We start with the multiple zeta function in $r$ variables $(r\geqslant 1)$,
the following generalization of the Riemann zeta function which is defined by 
$\displaystyle\zeta_{r}\big((s_{i})_{r}\big):=\sum_{0<m_{1}<\cdots<m_{r}} m_{1}^{-s_{1}} \cdots m_{r}^{-s_{r}}$
for $(s_{i})_{r}:= (s_{1},\ldots,s_{r}) \in \mathbb{C}^{r}$ such that $\text{Re}(s_{r-r'+1}+\cdots+s_{r})>r'$ for all $1\leqslant r' \leqslant r$. 
Its meromorphic continuation to the whole space $\mathbb{C}^{r}$ has been discussed in several papers, among whom \cite{AET,Z}.
In this paper, for a prime $p$, we consider the $p$-adic multiple $L$-function
$L_{p,r}\big((s_{i})_{r};(\omega^{k_{i}})_{r};(1)_{r};c\big)$,
a $p$-adic function which is defined  for $\big((s_{i})_{r};(k_{i})_{r}\big) \in \mathbb{C}_{p}^{r}\times  \mathbb{Z}^{r}$ with $|s_{i}|_{p} \leqslant p^{\frac{-1}{p-1}}$ for all $1\leqslant i \leqslant r$,
where $c\geqslant 2$  is a positive integer prime to $p$, and $\omega$ is the Teichm\"{u}ller character (cf. Definition \ref{definition of Lp,r}). 
It is introduced by the first-named author et al. (\cite{FKMT}),
and serves as  a $p$-adic analogue of the above $\zeta_{r}$
and also a generalization of the Kubota-Leopoldt $p$-adic $L$-function.

Multiple zeta values are the values of multiple zeta functions at tuples of positive integers, namely $\zeta\big((n_{i})_{r}\big)$ with $n_{i}$ positive integers and $n_{r}\geqslant 2$. They have been intensively studied over the last decade, as an example of periods, and with relations to knot theory, mathematical physics and other branches of mathematics. Cyclotomic multiple zeta values are a generalization of multiple zeta values which are also periods (\cite{Deligne Goncharov}) and depend additionally on a $r$-tuple $(\epsilon_{i})_{r}$ of roots of unity. One has $p$-adic analogues of multiple zeta values in \cite{Deligne Goncharov, F, F2} and, more generally, of cyclotomic multiple zeta values when the roots of unity under consideration are of order prime to $p$ in \cite{J1, U,Y}.

Cyclotomic multiple harmonic values (cf. Definition \ref{def mhs}) are a tool for studying $p$-adic cyclotomic multiple zeta values, recently introduced by the second-named author \cite{J3}, as sequences of weighted cyclotomic version of
multiple harmonic sums  of Hoffman \cite{H} and lifts of cyclotomic version of finite multiple zeta values of Kaneko-Zagier (unpublished). Given a positive integer $c$ prime to $p$, cyclotomic multiple harmonic values are certain explicit elements $\frak{H}\big((n_{i})_{r};(\epsilon_{i})_{r}\big)\in \prod_{p\in \mathcal{P}_{c}} \mathbb{Q}_{p}(\mu_{c})$, where $\mathcal{P}_{c}$ is the set of prime numbers which do not divide $c$, the $n_{i}$'s are positive integers and the $\epsilon_{i}$'s are in the group $\mu_c$ of $c$-th roots of unity ; one says that $((n_{i})_{r};(\epsilon_{i})_{r})$ have weight $\sum_{i=1}^{r}n_{i}$ and depth $r$.
For any $p \in \mathcal{P}_{c}$, the term indexed by $p$ in $\frak{H}\big((n_{i})_{r};(\epsilon_{i})_{r}\big)$ has $p$-adic valuation $\geqslant \sum_{i=1}^{r} n_{i}$.

The values $L_{p,r}((n_{i})_{r};(\omega^{-n_{i}})_{r};(1)_{r};c)$
with $n_{1},\ldots,n_{r}$ positive integers can be expressed 
as finite $\mathbb Q$-linear combinations of
the elements in $\mathbb Q_p (\mu_{pc})$
which are 
the values
of the $p$-adic twisted multiple star polylogarithm at tuples of $pc$-th roots of unity (\cite{FKMT}, Theorem 3.41).
In this paper, we give a different presentation:

\begin{Theorem} \label{the theorem}
For any $r$-tuple $(n_{i})_{r}$ of positive integers,
the family $\bigl(p^{\sum_{i=1}^{r}n_{i}}L_{p,r}((n_{i})_{r};(\omega^{-n_{i}})_{r};(1)_{r};c)\bigr)_{p \in \mathcal{P}_{c}}$ 
is expressed by the equation (\ref{eq:final equation}) as an infinite sum of series whose terms are $\mathbb{Q}(\mu_{c})$-linear combinations of cyclotomic multiple harmonic values with depth $\leqslant r$ and weight tending to infinity,
and whose convergence holds for the topology on $\prod_{p\in \mathcal{P}_{c}} \mathbb{Q}_{p}(\mu_{c})$ of the uniform convergence with respect to $p \in \mathcal{P}_{c}$.
\end{Theorem}

This theorem combined with the expression of cyclotomic multiple harmonic values in terms of $p$-adic cyclotomic multiple zeta values \cite{J2} gives an expression of $\bigl(p^{\sum_{i=1}^{r}n_{i}}L_{p,r}((n_{i})_r;(\omega^{-n_{i}})_r;(1)_r;c)\bigr)_{p \in \mathcal{P}_{c}}$ in terms of $p$-adic cyclotomic multiple zeta values, which are values of $p$-adic twisted multiple polylogarithms at tuples of $c$-th roots of unity. The sums of series of the above type are interpreted in terms of the Galois theory of $p$-adic cyclotomic multiple zeta values in \cite{J4}.

Our plan of this paper is as follows:
we recall the definitions of multiple $L$-functions $L_{p,r}$ and cyclotomic multiple harmonic values in \S \ref{Definitions} and then
we calculate the decomposition of the domain of the integration of $L_{p,r}$
in \S\ref{Decomposition of the domain of the integration}
and a variant of cyclotomic multiple harmonic sums
in \S\ref{Computation of a variant of cyclotomic multiple harmonic sums},
which are required to prove our theorem  \ref{the theorem}  in \S\ref{proof of the theorem}.

\emph{Acknowledgments.} The first-named author has been supported by JSPS KAKENHI 15KK0159 and 18H01110. The second-named author thanks Nagoya University for hospitality and has been supported by NCCR SwissMAP, Labex IRMIA, DAIKO Foundation and JSPS KAKENHI 18H01110. The authors thank the referee for useful comments.

\section{Definitions}\label{Definitions}
We review the definitions of the  main objects involved in this paper : the $p$-adic multiple $L$-functions $L_{p,r}$, cyclotomic multiple harmonic sums, and cyclotomic multiple harmonic values.

We denote by $\mathbb{N}$ the set of positive integers and by $\mathbb{N}_{0}=\mathbb{N} \cup \{0\}$. Let $p$ be a prime number.
Put $\mathcal{O}_{\mathbb{C}_{p}}$ to be the ring of integers of $\mathbb{C}_{p}$.
Let $\omega: \mathcal{O}_{\mathbb{C}_{p}}^{\times}\to  \mathcal{O}_{\mathbb{C}_{p}}^{\times}$ be the Teichmuller character,
and let $\langle x \rangle = \frac{x}{\omega(x)}$ for $x \in \mathcal{O}_{\mathbb{C}_{p}}^{\times}$.
Set 
$\int_{\mathbb{Z}_{p}} f(x) d \frak{m}_{z}(x) = \lim_{N \rightarrow \infty} \sum_{a=0}^{p^{N}-1} f(a) \frak{m}_{z}(a+p^{N}\mathbb{Z}_{p}).$ 
For $z \in \mathbb{P}^{1}(\mathbb{C}_{p})$ with $|z-1|_{p} \geqslant 1$, let $\frak{m}_{z}$ be the measure defined by
$\frak{m}_{z}(j+p^{N}\mathbb{Z}_{p})= \frac{z^{j}}{1-z^{p^{N}}}$  ($0\leqslant j \leqslant p^{N}-1$).
For $c \in \mathbb{N}$ prime to $p$, put
$\tilde{\frak{m}}_{c}:=\sum_{\substack{\xi^{c}=1\\\xi\not= 1}} \frak{m}_{\xi}$. 
Set $r\in\mathbb{N}$ and
$\frak{X}_{r}(d) := \{ (s_{1},\ldots,s_{r}) \in \mathbb{C}_{p}^{r}\text{ }|\text{ }|s_{j}|_{p} \leqslant d^{-1}p^{-\frac{1}{p-1}} \ (0\leqslant j \leqslant r)\}$.

\begin{Definition}[\cite{FKMT} Definition 1.16] \label{definition of Lp,r} 
Let $(s_i)_r:=(s_{1},\ldots,s_{r}) \in \frak{X}_{r}(q^{-1})$, $(k_i)_r:=(k_{1},\ldots,k_{r}) \in \mathbb{Z}^{r}$, $c \in \mathbb{N}_{> 1}$ which is prime to $p$, and
	\begin{equation} \label{eq: domain of integration} 
	(\mathbb{Z}_{p}^{r})' := 
	\{ (x_1,\dots, x_r) \in \mathbb{Z}_{p}^{r} \text{ }|\text{ }p \nmid x_{1},\ p \nmid (x_{1}+x_{2}),\ \ldots,\ p \nmid (x_{1}+\cdots+x_{r}) \}.
	\end{equation}
	The {\it $p$-adic multiple $L$-function} is defined by 
\begin{multline} L_{p,r} ((s_i)_r;(\omega^{k_{i}})_r;(1)_r;c)
\\ 
:= \int_{ (\mathbb{Z}_{p}^{r})'}
\langle x_{1}  \rangle^{-s_{1}} 
\langle x_{1}   + x_{2}  \rangle^{-s_{2}}
\cdots 	
\langle x_{1}  + \cdots + x_{r}  \rangle^{-s_{r}}
\omega^{k_{1}}(x_{1} )
\cdots 
\omega^{k_{r}}(x_{1} + \cdots + x_{r} )
\prod_{i=1}^{r} d\tilde{\frak{m}}_{c}(x_{j}).
\end{multline}
\end{Definition}

In \cite{FKMT} Example 1.19, it is explained that when $r=1$,
the Kubota-Leopoldt $p$-adic $L$-function is recovered: 
$$L_{p,1}(s;\omega^{k-1};\gamma_1;c)
=\langle \gamma_1\rangle^{-s} \omega^{k-1}(\gamma_1) (\langle c\rangle^{1-s}\omega^k(c)-1)\cdot L_p(s;\omega^k). 
$$

\begin{Definition} \label{def mhs}
	Let $(n_i)_r:=(n_{1},\ldots,n_{r}) \in \mathbb{N}^r$, and $(\epsilon_i)_r:=(\epsilon_{1},\ldots,\epsilon_{r})\in\mu_{c}^r$, and $m \in \mathbb{N}$.
	\newline (i) The {\it cyclotomic multiple harmonic sum} is the element in $\mathbb{Q}(\mu_{c})$
	defined by
\begin{equation*}\label{eqn:MHS}
H_{m}\left(
(n_{i})_{r};(\epsilon_{i})_{r}\right)
:= \sum_{0<m_{1}<\cdots<m_{r}<m} \frac{( \frac{\epsilon_{2}}{\epsilon_{1}})^{m_{1}} \cdots (\frac{1}{\epsilon_{r}})^{m_{r}}}{m_{1}^{n_{1}} \cdots m_{r}^{n_{r}}} \in \mathbb{Q}(\mu_{c}).
\end{equation*}
(ii) (\cite{J3}, Definition 1.3.1) Let $\mathcal{P}_{c}$ be the set of prime numbers which do not divide $c$. For $p \in \mathcal{P}_{c}$, 
we also denote by $\epsilon_{i}$ ($1\leqslant i\leqslant r$) the image of $\epsilon_{i}$ by the embedding $\mathbb{Q}(\mu_{c})\hookrightarrow \mathbb{Q}_{p}(\mu_{c})$. 
The {\it cyclotomic multiple harmonic value} is the family of multiple harmonic sums defined by
\begin{equation*}\label{eqn:MHV}
 \frak{H}\left((n_{i})_{r};(\epsilon_{i})_{r}\right) := \Bigl(  p^{n_{1}+\cdots+n_{r}}H_{p}\left((n_{i})_r;(\epsilon_{i})_r\right)\Bigr)_{p \in \mathcal{P}_{c}} \in  \prod_{p \in \mathcal{P}_{c}} \mathbb{Q}_{p}(\mu_{c}) .
 \end{equation*}
\end{Definition}

For differential forms $\eta_{1},\ldots,\eta_{n}$ in the set $\{\frac{dz}{z-z_{0}}\text{ }|\text{ }z_{0} \in \{0\}\cup \mu_{c}\}$ with $\eta_{1}\not=\frac{dz}{z}$, the formal iterated integral $I(\eta_{n},\ldots,\eta_{1}) \in K[[z]]$ is defined, by induction on $n$, by $I(\eta_{1}) = \int_{0}^{z} \eta_{1}$ and $I(\eta_{n},\ldots,\eta_{1})=\int_{0}^{z} I(\eta_{n-1},\ldots,\eta_{1})\eta_{n}$.
The cyclotomic multiple harmonic sums can be characterized in terms of iterated integrals as follows:
\begin{equation} \label{eq: MHS I 1} I(\omega_{0}^{l-1}\omega_{1}\omega_{0}^{n_{r}-1}\omega_{\epsilon_{r}}\cdots \omega_{0}^{n_{1}-1} \omega_{\epsilon_{1}}) = (-1)^{r+1} \sum_{0<m}H_{m}\left(
(n_{i})_{r} ; (\epsilon_{i})_{r} \right)\frac{z^{m}}{m^{l}}.
\end{equation}

\section{Decomposition of the domain of the integration of $L_{p,r}$}
\label{Decomposition of the domain of the integration}
We calculate the decomposition of the domain of the integration of $L_{p,r}((n_{i})_{r};(\omega^{-n_{i}})_{r};(1)_{r};c)$ which is required to prove our theorem  \ref{the theorem}  in \S\ref{proof of the theorem}.

For $i,j$ ($i<j$)  in $\mathbb{N}_0$, we put $[i,j]:=\{i,i+1,\dots ,j\}$.
By Definition \ref{definition of Lp,r},
any value of 
$L_{p,r} ((s_{i})_r;(\omega^{k_{i}})_r;(1)_r;c)$
at $(s_i)_r=(s_{1},\ldots,s_{r}) \in \frak{X}_{r}(q^{-1})$
is  given by the limit of finite sums indexed by the finite set
$$ D_{r,p^{M}} = \{(x_{i})_{r} \in [0,p^{M}-1]^{r} \ \bigm | \  p\nmid x_{1}+\cdots+x_{i}  \ (1\leqslant i \leqslant r)\} $$
\noindent with $M \in \mathbb{N}^{\ast}$ tending to $\infty$.
The goal of this section is to express $D_{r,p^{M}}$ in terms of domains of summation underlying multiple harmonic sums and variants. 
Let 
$$
\mathrm{div} : (x_{1},\dots,x_{r}) \in \mathbb{N}_0^{r} \mapsto 
\left((u_{1},\dots,u_{r});(t_{1},\dots,t_{r})\right) \in \mathbb{N}_0^{r} \times [0,p-1]^{r}
$$ be the map defined by the Euclidean division 
$x_{1}+\cdots+x_{i}=u_{i}p+t_{i}$ for all $i$ ($1\leqslant i \leqslant r$).

\begin{Lemma} \label{transformation domain of summation}
	For natural numbers $r$ and $M$,
	the above map $\mathrm{div}$ restricts to a bijection
	$$D_{r,p^{M}} \to  \amalg_{J \in E_{r}} \text{ } U_{r,p^{M},J} \times 
	T_{r,J},$$
	where $E_{r}$ is the set of triples $J=(P_{1},P_{2},P_{3})$ of subsets of $[1,r]$ such that $[1,r]=P_{1} \amalg  P_{2} \amalg P_{3}$, and, for $J=(P_{1},P_{2},P_{3}) \in E_r$, 
	\begin{align*}
	U_{r,p^{M},J} &:= \left\{(u_{1},\ldots,u_{r}) \in [0,rp^{M-1}-1]^{r}\text{ }\text{ } \left| 
	\begin{array}{l}
	i \in P_{1} \Rightarrow u_{i-1}=u_{i}
	\\ i \in P_{2} \Rightarrow u_{i-1}<u_{i}<u_{i-1}+p^{M-1}
	\\ i \in P_{3} \Rightarrow u_{i}=u_{i-1}+p^{M-1}
	\end{array} \right. \right\}, \\
	T_{r,J} &:= \left\{(t_{1},\ldots,t_{r}) \in [1,p-1]^r\text{ }\left|
	\text{ }\begin{array}{l}
	i \in P_{1} \Rightarrow t_{i-1} \leqslant t_{i}
	\\ i \in P_{3} \Rightarrow t_{i-1} > t_{i} 
	\end{array} \right. \right\}.
	\end{align*}
	Here we put $u_0=0$ and $t_0=0$.
\end{Lemma}

\begin{proof} 
	The map $(x_{1},\ldots,x_{r}) \mapsto (y_{1},\ldots,y_{r}) = (x_{1},x_{1}+x_{2},\ldots,x_{1}+\cdots+x_{r})$ sends the set 
	$D_{r,p^{M}}$
	bijectively to  $\{(y_{1},\ldots,y_{r}) \in \mathbb{N}^{r} \text{ }|\text{ } 0 < y_{1}\leqslant y_{2} \leqslant \ldots \leqslant y_{r}\text{ and } \forall i,\text{ }p\nmid y_{i}\text{ and } y_{i}-y_{i-1}< p^{M}\}$. If the Euclidean division of $y_{i}$ by $p$ is given by $y_{i} = pu_{i}+t_{i}$, we have the equivalences below:
\newline (a) $y_{i-1} \leqslant y_{i} \Leftrightarrow (u_{i-1}< u_{i})$ or $(u_{i-1}=u_{i}$ and $t_{i-1}\leqslant t_{i})$,
\newline (b) $y_{i} < y_{i-1}+p^{M} \Leftrightarrow$ ($u_{i}<u_{i-1}+ p^{M-1}$) or ($t_{i}< t_{i-1}$ and $u_{i}= u_{i-1}+ p^{M-1}$).
from which our result follows.
\end{proof}

Let $\Delta=\Delta_r$ be the set of {\it quasi-simplices}, where a quasi-simplex is a couple $(S,\Gamma)$ where $S=\{a_{1,1},\dots,
a_{1,l_1},a_{2,1},\dots,a_{2,l_2},\dots, a_{k,l_k}\}$ is a  (possibly
empty) subset of $\{1,\dots,r\}$, called the {\it support} of the
quasi-simplex, and $\Gamma=\emptyset$ or

\begin{equation}\label{presentation of quasi-simplex}
\Gamma=\{ (t_{a_{1,1}},\ldots,t_{a_{k,l_{k}}}) \in [1,p-1]^{S} \text{ }|\text{ } t_{a_{1,1}}=\cdots=t_{a_{1,l_1}} <t_{a_{2,1}}=\cdots=t_{a_{2,l_2}}<\cdots<t_{a_{k,1}}=\cdots= t_{a_{k,l_k}}\}
\end{equation}

We will frequently omit $(t_{a_{1,1}},\ldots,t_{a_{k,l_{k}}}) \in [1,p-1]^{S}$ in the notation. 
Let $\amalg\Delta$ be the set of couples $(S,\amalg_{i\in I} \Gamma_{i})$ where for all $i$, $(S,\Gamma_{i})$ is in $\Delta$. The quasi-simplices of support $S$ form a partition of $[1,p-1]^{S}$, thus, for an element $(S,\Gamma)$ of $\amalg\Delta$ the expression $\Gamma=\amalg_{i\in I} \Gamma_{i}$ with $(S,\Gamma_{i}) \in \Delta$ is unique.

\begin{Proposition} \label{lemma on domain on T}
Let $J=(P_{1},P_{2},P_{3})$ and $T_{r,J}$ be as in Lemma \ref{transformation domain of summation}. Let $P'_{i} = P_{i} \cup \{j \geqslant 1\text{ }|\text{ }j+1 \in P_{i}\}$, $i=1,3$. Then $(P'_{1}\cup P'_{3},T_{r,J}) \in \amalg \Delta$.
\end{Proposition}

\begin{proof} Let the {\it quasi-shuffle product} $\ast : \Delta \times \Delta \rightarrow \amalg \Delta$ (a.k.a, harmonic, stuffle product, \cite{H})
	be
	$(S_{1},\Gamma_1)\ast (S_{2},\Gamma_2) = (S_{1}\cup S_{2},(\Gamma_1 \times [1,p-1]^{S_{2}\setminus S_{1}}) \cap  (\Gamma_2 \times [1,p-1]^{S_{1}\setminus S_{2}}))$.
	For instance,$$
	(\{1\}, \{t_1\}) \ast (\{2\},\{t_2\})
	=(\{1,2\}, \big( \{t_1<t_2\}\amalg
	\{t_2<t_1\} \\ \amalg
	\{t_1=t_2\} \big)) $$
	Let the following variant $\tilde{\ast}$  of the quasi-shuffle product : denoting by $(S,\Gamma) = (S_{1},\Gamma_1)\ast (S_{2},\Gamma_2)$, we let
	$(S_{1},\Gamma_1)\tilde{\ast} (S_{2},\Gamma_2) :=
	(S,\Gamma \setminus (\Gamma_1<\Gamma_2))$, where  $\Gamma_1<\Gamma_2$ is the quasi-simplex obtained by combining the ``rightmost end`'' of $\Gamma_1$ with the ``leftmost end'' of $\Gamma_2$ in the presentation of \eqref{presentation of quasi-simplex}
	by the edge labeled by $<$. Namely,
	\begin{multline*}
	\{t_{a_{1,1}} = \cdots <\cdots< \cdots = t_{a_{k,l_{k}}}\} <
	\{t_{b_{1,1}} = \cdots <\cdots< \cdots = t_{b_{k',l_{k'}}}\}  
	\\ = 
	\{t_{a_{1,1}}= \cdots <\cdots< \cdots =t_{a_{k,l_{k}}}<t_{b_{1,1}}= \cdots <\cdots< \cdots = t_{b_{k',l_{k'}}}\}.
	\end{multline*}
	The products $\ast$ and $\tilde{\ast}$ extend to elements of $\amalg \Delta$. 	\newline\indent Let $\tilde{\Delta}=\tilde{\Delta}_r$ be the set defined like $\Delta$ except that the inequalities $t_{i}<t_{j}$ are replaced by the inequalities $t_{i} \leqslant t_{j}$.
	By following the rule $\{t_{i} \leqslant t_{j}\} = \{t_{i} = t_{j}\} \amalg \{t_{i} < t_{j}\} $,
	we have a natural inclusion
	$\phi : \tilde{\Delta} \hookrightarrow \amalg\Delta.$ 
	For instance, $\phi((\{1,2,3\}\times\{t_2\leqslant t_1=t_3\}))=(\{1,2,3\},\{t_2<t_1=t_3\} \amalg \{t_2=t_1=t_3\})$.
\newline\indent The {\it canonical increasing connected partition} of a subset $P \subset \{1,\ldots,r\}$ means the unique expression of $P$ as the disjoint union of the maximal (for the inclusion) segments:
$P=[i_{1},j_{1}]\amalg \cdots \amalg [i_{r'},j_{r'}]$ such that  $i_{1}\leqslant j_{1}<i_{2}\leqslant j_{2}< \cdots < i_{r'}\leqslant j_{r'}$ and $i_{s}-j_{s-1}\geqslant 2$ for all $s$ ($2\leqslant s \leqslant r'$). 
For each subset $[i,j]=\{i,i+1,\dots,j\}\subset\{1,\dots,r\}$, we define 
	$[i,j]_{1}:=([i,j],\{
	t_{i-1}\leqslant t_i\leqslant \cdots \leqslant t_j\})\in
	\tilde{\Delta}_r$
	and
	$[i,j]_{3}:=([i,j],\{ t_{j} < \cdots < t_{i} < t_{i-1} \})\in\Delta_r$. Let $C_{1}(P'_{1}\cup P'_{3})\amalg \cdots\amalg C_{u}(P'_{1}\cup P'_{3})$
be the canonical increasing connected partition of $P'_{1}\cup P'_{3}$.
For each $l$ ($1\leqslant l\leqslant u$), let
$$ [i'_{l,1},j'_{l,1}]\amalg \cdots \amalg [i'_{l,r'_{l}},j'_{l,r'_{l}}]
\quad \text{ and }\quad
[i_{l,1}'',j_{l,1}''] \amalg \cdots\amalg [i_{l,r_{l}''}'',j_{l,r_{l}''}''] $$
be the increasing connected partitions of $P'_{1} \cap C_{l}(P'_{1}\cup P'_{3})$ and
$P'_{3}\cap C_{l}(P'_{1}\cup P'_{3})$ respectively.
Then we have
\begin{equation} \label{eq:Proposition 5 explicit} (P'_{1}\cup P'_{3},T_{r,J}) =
\underset{1 \leqslant l \leqslant u}{\ast}
\left( \phi([i'_{l,1},j'_{l,1}]_{1})  \tilde{\ast} \cdots \tilde{\ast} \phi([i'_{l,r_{l}'},j'_{l,r_{l}'}]_{1}) \right) \text{ }
\ast \text{ } 
\left( [i_{l,r_{l}''}'',j_{l,r_{l}''}'']_{3} \tilde{\ast} \cdots \tilde{\ast} [i_{l,1}'',j_{l,1}'']_{3}
\right). 
\end{equation}
\end{proof}

\section{Computation of a variant of cyclotomic multiple harmonic sums}
\label{Computation of a variant of cyclotomic multiple harmonic sums}
We introduce and investigate a variant of cyclotomic multiple harmonic sums
which is necessary to prove the theorem \ref{the theorem}.

Let $c \in \mathbb{Z}_{p}^{\times}\cap\mathbb{N}$, $c\geqslant 2$. 
Put $(l_{i})_{r}=(l_{1},\ldots,l_{r})\in \mathbb{N}_{0}^{r}$
and $(\epsilon_i)_r=(\epsilon_{1},\ldots,\epsilon_{r})\in\mu_c^r$. 
Take $h \in \mathbb{N}$, and $(\kappa_{i})_{r-1}=(\kappa_{1},\dots,\kappa_{r-1})\in \mathbb{N}_{0}^{r-1}$. 
Let 
\begin{equation} \label{eq:variant 2 with gaps}
\mathcal{S}_{(\kappa_{i})_{r-1},h}((l_{i})_{r};(\epsilon_{i})_{r})
= \sum_{\substack{(u_{1},\ldots,u_{r}) \in \mathbb{N}^r_{0} \\ u_{1} < h
		\\ \forall i\geqslant 2,\text{ }u_{i-1}+\kappa_{i-1} h<u_{i}<u_{i-1}+(\kappa_{i-1}+1)h}}
 \big( \frac{\epsilon_{2}}{\epsilon_{1}}\big)^{u_{1}} \cdots \big(\frac{1}{\epsilon_{r}}\big)^{u_{r}} 
 u_{1}^{l_{1}} \cdots u_{r}^{l_{r}}\in \mathbb{Q}(\mu_{c}).
\end{equation}

The next lemma characterizes the dependence on $g$ of such functions.

\begin{Lemma} \label{mhs on U - 1}
For any $(l_{i})_{r} \in \mathbb{N}_{0}^{r}$, $(\epsilon_{i})_{r}\in \mu_{c}^{r}$, and $(\kappa_{i})_{r-1} \in \mathbb{N}_{0}^{r-1}$, there exists 
an element $\mathcal{B}_{l,\xi}^{((l_{i})_{r},(\epsilon_{i})_{r},(\kappa_{i})_{r-1})} \in \mathbb{Q}(\mu_{c})$ for each $l \in [0,l_{1}+\cdots l_{r}+r]$ and $\xi \in \mu_{c}$, such that, for all $h \in \mathbb{N}$, we have
\begin{equation}\label{eq: Hgap and B}
\mathcal{S}_{(\kappa_{i})_{r-1},h}((l_{i})_{r};(\epsilon_{i})_{r}) = \sum_{\substack{0 \leqslant l \leqslant  l_{1}+\cdots+l_{r}+r
		\\ \xi \in \mu_{c}}}\mathcal{B}_{l,\xi}^{((l_{i})_{r},(\epsilon_{i})_{r},(\kappa_{i})_{r-1})} h^{l} \xi^{h} .
	 \end{equation}	
	Moreover, we have $\displaystyle v_{p} \big( \mathcal{B}_{l,\xi}^{((l_{i})_{r},(\epsilon_{i})_{r},(\kappa_{i})_{r-1})}) \geqslant - r \{ 1+  \frac{ \log(l_{1}+\cdots+l_{r}+r)}{\log(p)}\}$.
\end{Lemma}

\begin{proof}
We prove the existence of the numbers $\mathcal{B}_{l,\xi}$ by induction on $r$.

Let us prove the claim for $r=1$. We have $\mathcal{S}_{\emptyset,h}(l_{1};\epsilon_{1}) = \sum\limits_{u_{1}=0}^{h-1} \epsilon_{1}^{-u_{1}} u_{1}^{l_{1}}$.  
\begin{itemize}
\item
If $\epsilon_{1}=1$, we write $\sum\limits_{u_{1}=0}^{h-1} u_{1}^{l_{1}} = \sum\limits_{l=1}^{l_{1}+1} \frac{1}{l_{1}+1} {l_{1}+1 \choose l} B_{l_{1}+1-l} h^{l}$ where $B$ denotes the Bernoulli numbers. 
This defines the coefficients
$\mathcal{B}_{l,\xi}^{(l_{1};1;\emptyset)}=\frac{1}{l_{1}+1} {l_{1}+1 \choose l} B_{l_{1}+1-l}$ if $\xi=1$ and $1 \leqslant l \leqslant l_{1}+1$ and 
$\mathcal{B}_{l,\xi}^{(l_{1};1;\emptyset)}=0$ otherwise.
By von Staudt-Clausen's theorem, for all $l_{1}$, we have $v_{p}(B_{l_{1}+1-l})\geqslant -1$ and $p^{v_{p}(l_{1}+1)} \leqslant l_{1}+1$ thus $v_{p}(\frac{1}{l_{1}+1}) \geqslant -\frac{\log(l_{1}+1)}{\log(p)}$.
Whence the desired bounds on $v_{p}(\mathcal{B}_{l,\xi}^{(l_{1};1;\emptyset)})$ follows.
\item
If $\epsilon_{1} \not=1$, let $\mathcal{E}$ be a formal variable.
We have $\sum\limits_{u_{1}=0}^{h-1} u_{1}^{l_{1}} \mathcal{E}^{u_{1}}= (\mathcal{E}\frac{d}{d\mathcal{E}})^{l_{1}} \left(\sum\limits_{u_{1}=0}^{h-1} \mathcal{E}^{u_{1}}\right) =  (\mathcal{E}\frac{d}{d\mathcal{E}})^{l_{1}}  (\frac{1-\mathcal{E}^{h}}{1-\mathcal{E}})$.
Consider a two-variable polynomial $R_{l_{1}}(x,y)$ with coefficients in $\mathbb{Z}[\mathcal{E}^{\pm 1}
]$ 
inductively constructed by 
$R_{l_1+1}(x,y)=(1-\mathcal{E})^{2^{l_1}}xy\frac{d}{dy}R_{l_1}(x,y)+2^{l_1}\mathcal{E}(1-\mathcal{E})^{2^{l_1}-1}R_{l_1}(x,y)$
and $R_0(x,y)=1-y$.
This fulfills $(\mathcal{E}\frac{d}{d\mathcal{E}})^{l_{1}}  (\frac{1-\mathcal{E}^{h}}{1-\mathcal{E}}) = \frac{R_{l_{1}}(h,\mathcal{E}^{h})}{(1-\mathcal{E})^{2^{l_{1}}}}$ for all $h \in \mathbb{N}$.
Since our $\mathcal{S}_{\emptyset,h}(l_{1};\epsilon_{1})$ is obtained 
by substituting $\frac{1}{\epsilon_{1}}$ to $\mathcal{E}$ there, 
it gets that $\mathcal{S}_{\emptyset,h}(l_{1};\epsilon_{1})$ is given by the finite linear combination of $h^l(\frac{1}{\epsilon_1^k})^h$  ($l,k\geqslant 0$) with coefficients in
$\mathbb{Z}[\epsilon_{1},\frac{1}{\epsilon_{1}},\frac{1}{\epsilon_{1}-1}]$.
By making an adjustment on $\frac{1}{\epsilon_1^k}$,
we get the presentation \eqref{eq: Hgap and B}.
It is easy to see
the coefficients $\mathcal{B}_{l,\xi}^{(l_{1};\epsilon_{1};\emptyset)}$ are all  in $\mathbb{Z}[\epsilon_{1},\frac{1}{\epsilon_{1}},\frac{1}{\epsilon_{1}-1}]$, which implies that their $p$-adic valuation is $\geqslant 0$ because of $|\epsilon_{1}|_{p} = |\epsilon_{1}-1|_{p}=1$.
\end{itemize}

Assume that our claim holds for $r$, and let us prove it for $r+1$. For any  $(l_{i})_{r+1} \in \mathbb{N}_{0}^{r+1}$, $(\epsilon_{i})_{r+1}\in \mu_{c}^{r+1}$, and $(\kappa_{i})_{r} \in \mathbb{N}_{0}^{r}$, we write
\begin{multline} \label{eq: lemma1} 
\mathcal{S}_{(\kappa_{i})_{r},h}((l_{i})_{r+1};(\epsilon_{i})_{r+1}) =
\\\sum_{\substack{(u_{1},\ldots,u_{r}) \in \mathbb{N}_{0}^r \\ u_{1} \leqslant \kappa_{1}h
\\ r\geqslant \forall i\geqslant 2,\text{ }u_{i-1}+\kappa_{i} h<u_{i}<u_{i-1}+(\kappa_{i}+1)h}}  \prod_{i=1}^{r} \big(\frac{\epsilon_{i+1}}{\epsilon_{i}}\big)^{u_{i}} u_{i}^{l_{i}} \sum_{u_{r}+\kappa_{r}h<u_{r+1}<u_{r}+\kappa_{r}h}  \big(\frac{1}{\epsilon_{r+1}}\big)^{u_{r+1}} u_{r+1}^{l_{r+1}}. 
\end{multline}
We have
$[u_{r}+\kappa_{r}h+1,u_{r}+(\kappa_{r}+1)h-1] = [0,u_{r}+(\kappa_{r}+1)h-1] - [0,u_{r}+\kappa_{r}h-1] - \{u_{r}+\kappa_{r}h\}$. Thus, by the result for $r=1$,
\begin{multline} \label{eq: lemma2}
\sum_{u_{r}+\kappa_{r}h<u_{r+1}<u_{r}+(\kappa_{r}+1)h}  \big(\frac{1}{\epsilon_{r+1}}\big)^{u_{r+1}} u_{r+1}^{l_{r+1}} =
\\ \sum_{\substack{0 \leqslant l \leqslant l_{r+1}+1 \\ \xi \in \mu_{c}}} \mathcal{B}_{l,\xi}^{(l_{r+1};\epsilon_{r+1};\emptyset)}\Big\{ \left(u_{r}+(\kappa_{r}+1\right)h)^{l} \xi^{(u_{r}+(\kappa_{r}+1)h)} -
	(u_{r}+\kappa_{r}h)^{l} \xi^{(u_{r}+\kappa_{r}h)} \Big\}\\ -\big(\frac{1}{\epsilon_{r+1}}\big)^{u_{r}+\kappa_{r}h}(u_{r}+\kappa_{r}h)^{l_{r+1}}. 
\end{multline}
By expanding
$(u_{r}+\left(\kappa_{r}+1)h\right)^{l}=\sum\limits_{\tilde{l}=0}^{l} {l \choose \tilde{l}} u_{r}^{\tilde{l}} (\kappa_{r}+1)^{l-\tilde{l}}h^{l-\tilde{l}}$, \
$(u_{r}+\kappa_{r}h)^{l}=\sum\limits_{\tilde{l}=0}^{l} {l \choose \tilde{l}} u_{r}^{\tilde{l}} \kappa_{r}^{l-\tilde{l}}h^{l-\tilde{l}}$, and
$(u_{r}+\kappa_{r}h)^{l_{r+1}}=\sum\limits_{\tilde{l}=0}^{l_{r+1}}{l_{r+1} \choose \tilde{l}}u_{r}^{\tilde{l}} (\kappa_{r}h)^{l_{r+1}-\tilde{l}}$ in \eqref{eq: lemma2},
we transform \eqref{eq: lemma1} to the following induction formula
    \begin{multline} \label{eq: lemma3} 
\mathcal{S}_{(\kappa_{i})_{r},h}((l_{i})_{r+1},(\epsilon_{i})_{r+1}) = \\ 
\sum_{\substack{0 \leqslant l \leqslant l_{r+1}+1 \\ \xi \in \mu_{c}\\ 0 \leqslant \tilde{l} \leqslant l}} 
	\mathcal{B}_{l,\xi}^{(l_{r+1};\epsilon_{r+1};\emptyset)}  {l \choose \tilde{l}}  h^{l-\tilde{l}} \Big\{
	(\kappa_{r}+1)^{l-\tilde{l}} \xi^{(\kappa_{r}+1)h}  - \kappa_{r}^{l-\tilde{l}} \xi^{\kappa_{r}h} \Big\} \mathcal{S}_{(\kappa_{i})_{r-1},h}((l_{i})_{r-1},l_{r}+\tilde{l};(\xi^{-1}\epsilon_{r}^{-1}\epsilon_{i})_{r})
	\\ - \sum_{\tilde{l}=0}^{l_{r+1}} {l_{r+1} \choose \tilde{l}} (\kappa_{r}h)^{l_{r+1}-\tilde{l}} \epsilon_{r+1}^{-\kappa_{r}h}  \mathcal{S}_{(\kappa_{i})_{r-1},h}((l_{i})_{r-1},l_{r}+\tilde{l};(\epsilon_{i})_{r}).
	\end{multline}
    The definition of the coefficients $\mathcal{B}_{l,\xi}^{((l_{i})_{r+1},(\epsilon_{i})_{r+1},(\kappa_{i})_{r})}$ are deduced from (\ref{eq: lemma3}) by applying the induction hypothesis to $\mathcal{S}_{(\kappa'_i)_{r-1},h}(l_{1},\ldots,l_{r-1},l_{r}+\tilde{l},\epsilon_{1},\ldots,\epsilon_{r-1},\epsilon_{r}\xi)$ and $\mathcal{S}_{(\kappa'_i)_{r-1},h}(l_{1},\ldots,l_{r-1},l_{r}+\tilde{l},\epsilon_{1},\ldots,\epsilon_{r-1},\epsilon_{r}\epsilon_{r+1})$.
    The bounds on their valuations is immediate by our induction hypotheses:
    \begin{align*}
    v_p\left(\mathcal{B}_{l,\xi}^{((l_{i})_{r+1},(\epsilon_{i})_{r+1}),(\kappa_{i})_{r}}) \right) &\geqslant
    -\left\{1+\frac{\log(l_{r+1}+1)}{\log p}\right\}
    -r\left\{1+\frac{ \log(l_{1}+\cdots+l_{r}+r)}{\log(p)}\right\} \\
    & \geqslant
    - (r+1) \left\{1+\frac{ \log(l_{1}+\cdots+l_{r+1}+r+1)}{\log(p)}\right\}.
    \end{align*}
\end{proof}

The next lemma expresses certain sums over the set $U_{r,p^{M},J}$ appearing in Lemma \ref{transformation domain of summation} in terms of the functions of (\ref{eq:variant 2 with gaps}).

\begin{Lemma} \label{mhs on U - 2}Let $J=(P_{1},P_{2},P_{3})$ be as in Lemma \ref{transformation domain of summation} (we take the convention that $1 \in P_{2}$).
For any $(l_{i})_{r} \in \mathbb{N}_{0}^{r}$, $(\epsilon_{i})_{r} \in \mu_{c}^{r}$, we have 
\begin{multline} \sum_{(u_{i})_{r} \in U_{r,p^{M},J}} \bigg(\frac{ \epsilon_{2}}{\epsilon_{1}}\bigg)^{u_{1}}\cdots \bigg(\frac{1}{\epsilon_{r}}\bigg)^{u_{r}} u_{1}^{l_{1}}\cdots u_{r}^{l_{r}} =
\\ 
\sum_{\substack{0\leqslant \tilde{l}_{j}\leqslant l^{(P_{2,3})}_{j} \\(j \in P_{3})}} \Big\{\prod_{j\in P_{3}}  {l^{(P_{2,3})}_{j} \choose \tilde{l}_{j}} \bigg(\frac{\epsilon_{j}^{(P_{2,3})}}{\epsilon_{j}}\bigg)^{\kappa_{j}p^{M-1}} (\kappa_{j}p^{M-1})^{l^{(P_{2,3})}_{j}-\tilde{l}_{j}} \Big\} \mathcal{S}_{(\kappa_{i})_{i \in P_{2}\setminus \{1\}},p^{M-1}}\bigg(\big( l^{(P_{2,3})}_{i}+\sum_{\substack{j \in P_{3} \\ j_{(P_{2})}=i}}\tilde{l}_{j} \big)_{i\in P_{2}};	(\epsilon_i)_{i\in P_2}\bigg).
\end{multline}
\end{Lemma}

Here, 
for $j \in P_{2}\cup P_{3}$, we put 
$\left\{\begin{array}{ll} l^{(P_{2,3})}_{j} = \sum_{k=j}^{\min((P_{2}\cup P_{3})\cap[j+1,r]) -1} l_{k} \\ \epsilon_{j}^{(P_{2,3})} = \epsilon_{\min((P_{2}\cup P_{3}) \cap [j+1,r])}\end{array}\right.$. \\
For $j \in P_{3}$, we put
$\left\{\begin{array}{ll} j_{(P_{2})} = \max[1,j-1]\cap P_{2}
\\ \kappa_{j} = \sharp\big(P_{3} \cap \big[j_{(P_{2})},j\big]\big)\end{array}\right.$.
And, for $i \in P_{2}$, we put
$\left\{\begin{array}{ll} i^{(P_{3})} = \min[i+1,r]\cap P_{3}
\\ \kappa_{i} = \sharp\big(P_{3} \cap \big[i,i^{(P_{3})}\big]\big)\end{array}\right.$.

\begin{proof} Put $\epsilon_{r+1}=1$.
	We have, for any $(u_{i})_{r} \in U_{r,p^{M},J}$,
	\begin{equation*} \begin{array}{cl} \displaystyle \prod_{i\in [1,r]}& \big(\frac{ \epsilon_{i+1}}{\epsilon_{i}}\big)^{u_{i}} u_{i}^{l_{i}}  = \displaystyle\prod_{i\in P_{2}\cup P_{3}} \bigg(\frac{ \epsilon_{i}^{(P_{2,3})}}{\epsilon_{i}}\bigg)^{u_{i}} u_{i}^{l^{(P_{2,3})}_{i}}
	\\ & = \displaystyle\prod_{i\in P_{2}} \big(\frac{ \epsilon_{i}^{(P_{2,3})}}{\epsilon_{i}}\big)^{u_{i}} u_{i}^{l^{(P_{2,3})}_{i}} \prod_{j \in P_{3}} \bigg(\frac{ \epsilon_{j}^{(P_{2,3})}}{\epsilon_{j}}\bigg)^{u_{j_{(P_{2})}}+\kappa_{j}p^{M-1}} (u_{j_{(P_{2})}}+ \kappa_{j}p^{M-1})^{l^{(P_{2,3})}_{j}} 
	\\ & = \displaystyle 
	\prod_{i\in P_{2}} \bigg(\frac{ \epsilon_{i}^{(P_{2,3})}}{\epsilon_{i}}\bigg)^{u_{i}} u_{i}^{l^{(P_{2,3})}_{i}} \prod_{j \in P_{3}} \bigg(\frac{ \epsilon_{j}^{(P_{2,3})}}{\epsilon_{j}}\bigg)^{u_{j_{(P_{2})}}+\kappa_{j}p^{M-1}} \bigg\{
	\sum_{\tilde{l}_{j}=0}^{l_{j}^{(P_{2,3})}}
	{l_{j}^{(P_{2,3})} \choose \tilde{l}_{j}} u_{j_{(P_{2})}}^{\tilde{l}_{j}}(\kappa_{j}p^{M-1})^{l_{j}^{(P_{2,3})}-\tilde{l}_{j}} 
	\bigg\} .
	\end{array} 
	\end{equation*}
For $i \in P_{2}$, let $i^{(P_{2})}= \min(P_{2}\cap [i+1,r])$ and $\epsilon_{i}^{(P_{2})}=\epsilon_{i^{(P_{2})}}$. We have
	$$ \prod_{i\in P_{2}} \bigg(\frac{ \epsilon_{i}^{(P_{2,3})}}{\epsilon_{i}}\bigg)^{u_{i}} \prod_{j \in P_{3}} \bigg(\frac{ \epsilon_{j}^{(P_{2,3})}}{\epsilon_{j}}\bigg)^{u_{j_{(P_{2})}}} = 	
	\prod_{i\in P_{2}} \bigg\{\frac{ \epsilon_{i}^{(P_{2,3})}}{\epsilon_{i}} \bigg(\prod_{j \in P_{3},j_{(P_{2})}=i}  \frac{ \epsilon_{j}^{(P_{2,3})}}{\epsilon_{j}}\bigg)\bigg\}^{u_{i}} = \prod_{i \in P_{2}} \bigg( \frac{\epsilon_{i}^{(P_{2})}}{\epsilon_{i}} \bigg)^{u_{i}} .$$
We obtain the result by summing over all $(u_{i})_{r} \in U_{r,p^{M},J}$.
\end{proof}

\section{End of the proof of theorem \ref{the theorem}}\label{proof of the theorem}
We finish the proof of Theorem \ref{the theorem}.

\begin{proof} Let $(n_{i})_{r} \in \mathbb{N}^{r}$. By the definition of $L_{p,r}$ and by the definition of $D_{r,p^{M}}$ (\S2.1), we have 
$$ L_{p,r}((n_{i})_{r};(\omega^{-n_{i}})_{r};(1)_{r};c) = \lim_{M\rightarrow \infty} \sum_{(x_{i})_{r} \in D_{r,p^{M}}}
\prod_{i=1}^{r}\Big\{\frac{1}{(x_{1}+\cdots+x_{i})^{n_{i}}}
\sum_{\substack{\xi_{i} \in \mu_{c} \setminus \{1\}}}\frac{\xi_{i}^{x_{i}}}{1-\xi_{i}^{p^{M}}} \Big\}. $$
Put $M \in \mathbb{N}$ and $(x_{i})_{r} \in D_{r,p^{M}}$.
Let $((u_{i})_{r};(t_{i})_{r})$ be its image by the map $\mathrm{div}$. For all $i \in [1,r]$, we write 
$(x_{1}+\cdots+x_{i})^{-n_{i}}=(pu_{i}+t_{i})^{-n_{i}} = \sum\limits_{l_{i}\in \mathbb{N}_{0}} {-n_{i} \choose l_{i}} t_{i}^{-n_{i}} \big(\frac{p u_{i}}{t_{i}} \big)^{l_{i}}$.
Then, by Lemma \ref{transformation domain of summation},
\begin{multline} \label{eq:end1}
L_{p,r}((n_{i})_{r};(\omega^{-n_{i}})_{r};(1)_{r};c) = 
\\
 \lim_{M\rightarrow\infty} \sum_{\substack{{\bf l}=(l_{i})_{r} \in \mathbb{N}_{0}^{r}
\\ (\xi_{i})_{r} \in (\mu_{c} \setminus \{1\})^{r}} }  \prod_{i=1}^{r} \frac{  {-n_{i} \choose l_{i}}}{1 - \xi_{i}^{p^{M}}} \sum_{J \in E_{r}}\sum_{(u_{i})_{r} \in U_{r,p^{M},J}}
\prod_{i=1}^{r}\big( p u_{i}\big)^{l_{i}} \xi_{i}^{pu_{i}-pu_{i-1}} \sum_{(t_{i})_{r} \in T_{r,J}} 
\prod_{i=1}^{r}\frac{\xi_{i}^{t_{i}-t_{i-1}}}{t_{i}^{n_{i}+l_{i}}}. 
\end{multline}
\newline Let $J=(P_{1},P_{2},P_{3}) \in E_{r}$.
\begin{itemize}
\item Let $\Delta(T_{r,J})$ be the set of the quasi-simplices
appearing in the decomposition of $T_{r,J}$ in the equation (\ref{eq:Proposition 5 explicit}), so we have $T_{r,J} = \amalg_{\delta\in
	\Delta(T_{r,J})} \delta$.
For each quasi-simplex $\delta\in \Delta(T_{r,J})$ with the presentation
\eqref{presentation of quasi-simplex},
the index ${\bf l}=(l_i)_r\in \mathbb N_0^r$,
and the pair ${\bf w}:=((n_i)_r;(\xi_i)_r)\in \mathbb{N}^{r}\times \mu_c^r$,
we set ${\bf w}({\bf l})_\delta:=({\bf n}({\bf l})_\delta;\xi({\bf l})_\delta)$
with
${\bf n}({\bf l})_\delta=(\sum\limits_{j=1}^{l_1} (n_{a_{1,j}}+l_{a_{1,j}}),
\dots,\sum\limits_{j=1}^{l_k} (n_{a_{k,j}}+l_{a_{k,j}}))$
and
$\xi({\bf l})_\delta=(\prod\limits_{i=1}^k\prod\limits_{j=1}^{l_i}\frac{\xi_{a_{i,j}}}{\xi_{a_{i,j}+1}},\prod\limits_{i=2}^k\prod\limits_{j=1}^{l_i}\frac{\xi_{a_{i,j}}}{\xi_{a_{i,j}+1}},\dots,
\prod\limits_{j=1}^{l_k}\frac{\xi_{a_{k,j}}}{\xi_{a_{k,j}+1}})$.
Then the equation (\ref{eq:Proposition 5 explicit}) implies
\begin{equation} \label{eq:end lemma T}
\sum_{(t_{i})_{r} \in T_{r,J}} \prod_{i=1}^{r}
\frac{\xi_{i}^{t_{i}-t_{i-1}}}{t_{i}^{n_{i}+l_{i}}}  = \sum_{\delta
	\in \Delta(T_{r,J})} H_{p}({\bf w}({\bf l})_{\delta}).
\end{equation}

\item By applying Lemma \ref{mhs on U - 2} then Lemma \ref{mhs on U - 1}, we have, for all  $(l_{i})_{r} \in \mathbb{N}_{0}^{r}$ and $(\xi_{i})_{r} \in (\mu_{c} \setminus \{1\})^{r}$,
\begin{multline} \label{eq:end lemma U}
\sum_{(u_{i})_{r} \in U_{r,p^{M},J}} \prod_{i=1}^{r}
u_{i}^{l_{i}} \xi_{i}^{pu_{i}-pu_{i-1}} 
\\ =
\sum_{\substack{0\leqslant \tilde{l}_{j}\leqslant l^{(P_{2,3})}_{j} \\(j \in P_{3})}} \Big\{\prod_{j\in P_{3}} (\kappa_{j}p^{M-1})^{l^{(P_{2,3})}_{j}-\tilde{l}_{j}} \bigg(\frac{(\xi_{j}^{(P_{2,3})})}{\xi_{j}}\bigg)^{(-p\kappa_{j})p^{M-1}} {l^{(P_{2,3})}_{j} \choose \tilde{l}_{j}} \Big\}
\\ .\text{ }\mathcal{S}_{((\kappa_{i})_{i \in P_{2}},p^{M-1})}\big(\big( l^{(P_{2,3})}_{i}+\sum_{\substack{j\in P_{3}\\ j_{(P_{2})}=i}}\tilde{l}_{j} \big)_{i\in P_{2}};	
(\xi^{-p}_i)_{i\in P_2}\big)
\\ =
\sum_{\substack{0\leqslant \tilde{l}_{j}\leqslant l^{(P_{2,3})}_{j} \\(j \in P_{3})}} \Big\{\prod_{j\in P_{3}} (\kappa_{j}p^{M-1})^{l^{(P_{2,3})}_{j}-\tilde{l}_{j}} \bigg(\frac{(\xi_{j}^{(P_{2,3})})}{\xi_{j}}\bigg)^{(-p\kappa_{j})p^{M-1}} {l^{(P_{2,3})}_{j} \choose \tilde{l}_{j}} \Big\}
\\
 . \sum_{\substack{0 \leqslant l \leqslant l_{1}+\cdots+l_{r}+r
		\\ \xi \in \mu_{c}}}\mathcal{B}_{l,\xi}^{(( l^{(P_{2,3})}_{i}+\sum_{\substack{j\in P_{3}\\ j_{(P_{2})}=i}}\tilde{l}_{j})_{i\in P_{2}},(\xi^{-p}_i)_{i\in P_2},(\kappa_{i})_{i \in P_{2}\setminus \{1\}})} (p^{M-1})^{l} \xi^{p^{M-1}}.
\end{multline}
\end{itemize}
\indent Substituting the equations \eqref{eq:end lemma T} and \eqref{eq:end lemma U} in the equation \eqref{eq:end1}, and multiplying by $p^{n_{1}+\cdots+n_{r}}$, we obtain

\begin{multline} \label{eq:end 2}
p^{n_{1}+\cdots+n_{r}}L_{p,r}((n_{i})_{r};(\omega^{-n_{i}})_{r};(1)_{r};c) = 
\\
\lim_{M\rightarrow\infty} \sum_{\substack{{\bf l}=(l_{i})_{r} \in \mathbb{N}_{0}^{r}
		\\ (\xi_{i})_{r} \in (\mu_{c} \setminus \{1\})^{r} 
		\\ J=(P_{1},P_{2},P_{3}) \in E_{r}}}  \prod_{i=1}^{r} \frac{  {-n_{i} \choose l_{i}}}{1 - \xi_{i}^{p^{M}}} 
\sum_{\substack{0\leqslant \tilde{l}_{j}\leqslant l^{(P_{2,3})}_{j} \\(j \in P_{3})}} \Big\{\prod_{j\in P_{3}} (\kappa_{j}p^{M-1})^{l^{(P_{2,3})}_{j}-\tilde{l}_{j}} \bigg(\frac{(\xi_{j}^{(P_{2,3})})}{\xi_{j}}\bigg)^{(-p\kappa_{j})p^{M-1}} {l^{(P_{2,3})}_{j} \choose \tilde{l}_{j}} \Big\} 
\\
\cdot\sum_{\substack{0 \leqslant l \leqslant  l_{1}+\cdots+l_{r}+r
		\\ \xi \in \mu_{c}}}\mathcal{B}_{l,\xi}^{(( l^{(P_{2,3})}_{i}+\sum_{\substack{j\in P_{3}\\ j_{(P_{2})}=i}}\tilde{l}_{j})_{i\in P_{2}},(\xi^{-p}_i)_{i\in P_2},(\kappa_{i})_{i \in P_{2}\setminus \{1\}})} (p^{M-1})^{l} \xi^{p^{M-1}}
\sum_{\delta \in \Delta(T_{r,J})} p^{\weight({\bf w}({\bf l})_{\delta})} H_{p}({\bf w}({\bf l})_{\delta}).
\end{multline}
\indent For any $b \in \mathbb{N}$,$\epsilon\in \mu_{c}$, $ (\xi_{i})_{r} \in (\mu_{c} \setminus \{1\})^{r}$  we define $A_{b,\epsilon,(\xi_{i})_{r}}^{(n_{i})_{r}}$ below, which does not depend on $M$ :
\begin{align} \label{eq:reordered 1}
A_{b,\epsilon,(\xi_{i})_{r}}^{(n_{i})_{r}}
&=
 \sum_{J=(P_{1},P_{2},P_{3}) \in E_{r}}
	\sum_{\substack{{\bf l}=(l_{i})_{r} \in \mathbb{N}_{0}^{r}
		\\ 0\leqslant \tilde{l}_{j} \leqslant l_{j}^{(P_{2,3})} (j \in P_{3}) 
		\\ l + \sum_{j \in P_{3}} l_{j}^{(P_{2,3})} - \tilde{l}_{j} = b}}
    \sum_{\substack{ \xi \in \mu_{c}
		\\ \prod_{j \in P_{3}}(\frac{\xi_{j}^{(P_{2,3})}}{\xi_{j}})^{-\kappa_{i}p}
		\xi = \epsilon}}   
		\prod_{i=1}^{r} {-n_{i} \choose l_{i}} \Big\{
		\prod\limits_{j\in P_{3}} \kappa_{j}^{l^{(P_{2,3})}_{j}-\tilde{l}_{j}}  {l^{(P_{2,3})}_{j} \choose \tilde{l}_{j}} \Big\}
		\\  \notag
\cdot&
\mathcal{B}_{l,\xi}^{(( l^{(P_{2,3})}_{i}+\sum_{\substack{j\in P_{3}\\ j_{(P_{2})}=i}}\tilde{l}_{j})_{i\in P_{2}},(\xi^{-p}_i)_{i\in P_2},(\kappa_{i})_{i \in P_{2}\setminus \{1\}})}
\sum_{\delta \in \Delta(T_{r,J})} p^{\weight({\bf w}({\bf l})_{\delta})} H_{p}({\bf w}({\bf l})_{\delta}).
\end{align}
The convergence of the right-hand side of (\ref{eq:reordered 1}) follows from that it is a sum of series over
$(l_{i})_{r} \in \mathbb{N}_{0}^{r}$ whose general term has valuation $\displaystyle\geqslant n_{1}+\cdots+n_{r}+l_{1}+\cdots+l_{r} - r\{1+ \frac{\log(l_{1}+\cdots+l_{r}+r)}{\log(p)}\}$, which tends to $\infty$ when $\sum\limits_{i=1}^{r} l_{i} \rightarrow \infty$. Indeed, for all $(m_{i})_{r} \in \mathbb{N}^{r}$ and $(\epsilon_{i})_{r} \in \mu_{c}^{r}$, we have clearly
$$ v_{p} \left( p^{m_{1}+\cdots+m_{r}} H_{p}\left( (m_{i})_r,(\epsilon_{i})_r \right) \right) \geqslant m_{1} + \cdots + m_{r} $$
and for all $J \in E_{r}$, $(l_{i})_{r} \in \mathbb{N}_{0}^{r}$ and $(\tilde{l}_{i})_{i\in P_{2}} \in \prod_{i \in P_{2}} [0,l_{i}^{(P_{2,3})}]$, the first inequality below follows from Lemma \ref{mhs on U - 1}, and the second one follows from $\sharp P_{2} \leqslant r$ and Lemma \ref{mhs on U - 2}.
$$
\begin{array}{ll} \displaystyle v_{p}\Biggl(\mathcal{B}^{(( l^{(P_{2,3})}_{i}+\sum_{\substack{j\in P_{3}\\ j_{(P_{2})}=i}}\tilde{l}_{j})_{i\in P_{2}},(\xi^{-p}_i)_{i\in P_2},(\kappa_{i})_{i \in P_{2}\setminus \{1\}})}_{l,\xi}\Biggr) 
& \displaystyle\geqslant - \sharp P_{2} \bigg\{1 + \frac{\log( \sum\limits_{i \in P_{2}} l_{i}^{(P_{2,3})}+ \sum\limits_{j \in P_{3}} \tilde{l}_{j}+\sharp P_{2})}{\log(p)}\bigg\},
\\ & \displaystyle\geqslant - r \bigg\{1 + \frac{\log( \sum\limits_{i=1}^{r}l_{i}+r)}{\log(p)}\bigg\}.
\end{array}
$$
Reordering the equation (\ref{eq:end 2}), we obtain
\begin{equation} \label{eq:reordered 2}
p^{n_{1}+\cdots+n_{r}}L_{p,r}((n_{i})_r;(\omega^{-n_{i}})_r;(1)_r;c) =  \lim_{M\rightarrow \infty}
\sum_{\substack{0 \leqslant b \\ \epsilon \in \mu_{c}
		\\(\xi_{i})_{r} \in (\mu_{c} \setminus \{1\})^{r}} }  
\frac{(p^{M-1})^{b} \epsilon^{p^{M-1}}}{\prod_{i=1}^{r}(1 - \xi_{i}^{p^{M}})} A_{b,\epsilon,(\xi_{i})_{r}}^{(n_{i})_{r}}.
\end{equation}
The convergence of (\ref{eq:reordered 2}) is proved as follows:
$$
\begin{array}{ll}v_{p}( A_{b,\epsilon,(\xi_{i})_{r}}^{(n_{i})_{r}}) & \geqslant n_{1}+\cdots+n_{r} + \inf \bigg\{ l_{1}+\cdots+l_{r}- r\left\{ 1 + \frac{\log(l_{1}+\cdots+l_{r}+r)}{\log(p)}\right\} \text{ }|\text{ } (l_{i})_{r} \in \mathbb{N}_{0}^{r} \bigg\}
\\ & \geqslant n_{1}+\cdots+n_{r} + 
\min \bigg\{ l_{1}+\cdots+l_{r}- r\left\{ 1 + \frac{\log(l_{1}+\cdots+l_{r}+r)}{\log(p)}\right\} \text{ }|\text{ } (l_{i})_{r} \in \mathbb{N}_{0}^{r}, \sum\limits_{i=1}^{r} l_{i} \in [0,3r] \bigg\}
\\ & \geqslant n_{1}+\cdots+n_{r} - r\left\{ 1 + \frac{\log(4r)}{\log(p)}\right\}.
\end{array} $$
Here the first inequality follows from the above discussion ; the second inequality follows from the fact that the function $f: t \in (-r,\infty) \mapsto t - r \{1 + \frac{\log(t+r)}{\log(p)}\}\in \mathbb{R}$ is increasing on $[\frac{r}{\log(p)}-r,\infty)$ which contains the interval $[3r,\infty)$ ; the third inequality follows from the fact that, for $ (l_{i})_{r} \in\mathbb{N}_{0}^{r}$ such that $\sum\limits_{i=1}^{r} l_{i} \in [0,3r]$ we have $\sum\limits_{i=1}^{r} l_{i}\geqslant 0$ and $-\log(\sum\limits_{i=1}^{r} l_{i}+r)\geqslant -\log(4r)$.

We can compute the limit in (\ref{eq:reordered 2}) by restricting $M$ to the case with $p^{M-1}\equiv 1 \mod c$.
For such $M$'s, we have $\xi^{p^{M-1}}=\xi$ for all $\xi \in \mu_{c}$, and, since $p^{M-1}\to 0$ when ${M\to\infty}$, the limit in \eqref{eq:reordered 2}
is equal to 
$$
\sum_{\substack{ \epsilon \in \mu_{c}
		\\(\xi_{i})_{r} \in (\mu_{c} \setminus \{1\})^{r}}} 
\frac{\epsilon}{\prod_{i=1}^{r}(1 - \xi_{i}^{p})} A^{(n_{i})_{r}}_{0,\epsilon,(\xi_{i})_{r}}.
$$
This is true for any $p \in \mathcal{P}_{c}$. Thus, we deduce :

\begin{multline*}
\left(p^{n_{1}+\cdots+n_{r}}L_{p,r}((n_{i})_r;(\omega^{-n_{i}})_r;(1)_r;c)\right)_{p \in \mathcal{P}_{c}} =
\\ 
\sum_{{\bf l}=(l_{i})_{r} \in \mathbb{N}_{0}^{r}}
\sum_{\substack{ \epsilon \in \mu_{c}
		\\(\xi_{i})_{r} \in (\mu_{c} \setminus \{1\})^{r}}} 
\frac{\epsilon}{\prod_{i=1}^{r}(1 - \xi_{i}^{p})}
\sum_{J=(P_{1},P_{2},P_{3}) \in E_{r}}
\sum_{\substack{ \xi \in \mu_{c}
\\ \prod_{j \in P_{3}}(\frac{\xi_{j}^{(P_{2,3})}}{\xi_{j}})^{-\kappa_{i}p}
\xi = \epsilon}}   
\prod_{i=1}^{r} {-n_{i} \choose l_{i}}
\\  
\cdot\mathcal{B}_{0,\xi}^{(( l^{(P_{2,3})}_{i}+\sum_{\substack{j\in P_{3}\\ j_{(P_{2})}=i}}l^{(P_{2,3})}_{j})_{i\in P_{2}},(\xi^{-p}_i)_{i\in P_2},(\kappa_{i})_{i \in P_{2}\setminus \{1\}})}
\sum_{\delta \in \Delta(T_{r,J})}  \frak{H}({\bf w}({\bf l})_{\delta}).
\end{multline*}

By adjusting $p$-th powers,
we reformulate it to be
\begin{multline} \label{eq:final equation} 
\left(p^{n_{1}+\cdots+n_{r}}L_{p,r}((n_{i})_r;(\omega^{-n_{i}})_r;(1)_r;c)\right)_{p \in \mathcal{P}_{c}} =
\\ 
\sum_{{\bf l}=(l_{i})_{r} \in \mathbb{N}_{0}^{r}}
\sum_{\substack{ \epsilon \in \mu_{c}
		\\(\xi_{i})_{r} \in (\mu_{c} \setminus \{1\})^{r}}} 
\frac{\epsilon}{\prod_{i=1}^{r}(1 - \xi_{i})}
\sum_{J=(P_{1},P_{2},P_{3}) \in E_{r}}
\sum_{\substack{ \xi \in \mu_{c}
\\ \prod_{j \in P_{3}}(\frac{\xi_{j}^{(P_{2,3})}}{\xi_{j}})^{-\kappa_{i}}
\xi = \epsilon}}   
\prod_{i=1}^{r} {-n_{i} \choose l_{i}}
\\  
\cdot\mathcal{B}_{0,\xi}^{(( l^{(P_{2,3})}_{i}+\sum_{\substack{j\in P_{3}\\ j_{(P_{2})}=i}}l^{(P_{2,3})}_{j})_{i\in P_{2}},(\xi^{-1}_i)_{i\in P_2},(\kappa_{i})_{i \in P_{2}\setminus \{1\}})}
\sum_{\delta \in \Delta(T_{r,J})}  \frak{H}({\bf w}({\bf l})_{\delta})^{\Frob^{-1}}.
\end{multline}
Here  for $(x_p)_p\in\prod_{p\in \mathcal{P}_{c}} \mathbb{Q}_{p}(\mu_{c})$,
we define  $(x_p)_p^{\Frob^{-1}}$ to be $(\Frob_p^{-1}(x_p))_p$
with the Frobenius automorphism $\Frob_p:\mathbb{Q}_{p}(\mu_{c})\to\mathbb{Q}_{p}(\mu_{c})$ sending $\xi\mapsto\xi^p$ for $\xi\in\mu_c$.

Finally, in the right-hand side of (\ref{eq:final equation}), the term indexed by any ${\bf l}=(l_{i})_{r} \in \mathbb{N}_{0}^{r}$ has valuation 
$\displaystyle\geqslant n_{1}+\cdots+n_{r}+l_{1}+\cdots+l_{r} - r\{1+ \frac{\log(l_{1}+\cdots+l_{r}+r)}{\log(p)}\}$, 
which tends to $\infty$ when $\sum\limits_{i=1}^rl_{i} \rightarrow \infty$, uniformly with respect to $p$.
\end{proof}

\begin{Example} \label{example}
For $r=1$, our equation (\ref{eq:final equation}) is 
$$ \big( p^{n}L_{p}(n,\omega^{-n},1,c) \big)_{p\in\mathcal{P}_{c}} = \sum_{l\geqslant 0}{-n \choose l} \sum_{\xi_{1} \in \mu_{c} \setminus \{1\}} \sum_{\epsilon\in \mu_{c}} \mathcal{B}_{0,\epsilon}^{(l,\xi_{1}^{-1},\emptyset)}\frac{\epsilon}{1-\xi_{1}} \frak{H}(n+l,\xi_{1})^{\Frob^{-1}} .$$
This formula is a variant of the  following formula 
 which is a particular case of \cite{W}, Theorem 5.11:
 $$
\left(p^mL_p(m;\omega^{1-m})\right)_p
=\sum_{s\geqslant m-1}\frac{(-1)^{s+m+1}}{m-1}\binom{s-1}{m-2}
B_{s+1-m}\cdot{\frak H}(s).$$
\end{Example}

\begin{Remark}
A generalization $L_{p,r,\alpha}$ of $L_{p,r}$ with $\alpha \in \mathbb{N}$ can be defined by replacing the condition $p\nmid x_{1}\gamma_{1}+\ldots+x_{i}\gamma_{i}$ by the condition $p^{\alpha}\nmid x_{1}\gamma_{1}+\ldots+x_{i}\gamma_{i}$ ($1\leqslant i \leqslant r$) in the equation (\ref{eq: domain of integration}).
The results of \cite{FKMT} and Theorem \ref{the theorem} can be generalized to $L_{p,r,\alpha}$ by similar proofs, provided $p^{n_{1}+\cdots+n_{r}}H_{p}\left((n_{i})_r;(\epsilon_{i})_r\right)$ is replaced by $p^{\alpha(n_{1}+\cdots+n_{r})}H_{p^{\alpha}}\left((n_{i})_r;(\epsilon_{i})_r\right)$ in Definition \ref{def mhs}, which defines a more general type of cyclotomic multiple harmonic values satisfying the same properties (\cite{J1,J2,J3,J4}).
\end{Remark}

\end{document}